\documentclass[a4paper,11pt]{amsart}
\usepackage{amsfonts}
 \usepackage{enumerate}
\usepackage{amsmath, amsthm}
\usepackage{amscd}
\usepackage{amssymb}
\input xy
\xyoption{all}
\usepackage{latexsym,graphicx}
\usepackage[latin1]{inputenc}
\usepackage[cyr]{aeguill}
\usepackage{xspace}
\usepackage{array}
\usepackage{newclude}
\usepackage{hyperref}
\hypersetup{pdfpagelabels,
                  plainpages=false,
                    colorlinks=true,       
                    linkcolor=red}
\renewcommand*{\HyperDestNameFilter}[1]{\jobname-#1} 
\numberwithin{equation}{section}
\usepackage[centering]{geometry}

\usepackage[nameinlink]{cleveref}

\newcommand{\sspace}{\vspace{0.25cm}}

 \theoremstyle{plain}
\newtheorem{theor}{Theorem}[section]

\newtheorem{prop}[theor]{Proposition}
\newtheorem{lem}[theor]{Lemma}
\newtheorem{sublem}[theor]{Sublemma}

\theoremstyle{remark}
\newtheorem{rem}[theor]{Remark}
\newtheorem{rems}[theor]{Remarks}
\newtheorem{Example}[theor]{Example}

\theoremstyle{plain}
\newtheorem{defi}[theor]{Definition}

\numberwithin{equation}{section}

\newcommand{\CC}{{\mathbb C}}
\newcommand{\RR}{{\mathbb R}}
\renewcommand{\SS}{{\mathbf S}}
\newcommand{\QQ}{{\mathbb Q}}
\newcommand{\FF}{{\mathcal{ F}}}
\newcommand{\ZZ}{{\mathbb Z}}
\newcommand{\VV}{{\mathbb V}}

\newcommand{\G}{{\mathbf G}}
\newcommand{\HH}{{\mathbf H}}

\newcommand{\LL}{{\mathbf L}}

\newcommand{\PP}{{\mathbf P}}

\newcommand{\NN}{{\mathbb N}}

\newcommand{\Ga}{\Gamma}

\newcommand{\Gr}{{\textnormal{Gr}}}
\newcommand{\HL}{\textnormal{HL}}

\newcommand{\ti}[1]{\mbox{$\tilde{#1} $}}

\newcommand{\ol}{\overline}

\newcommand{\lo}{\longrightarrow}

\newcommand{\End}{{\rm End}\,}

\newcommand{\ad}{{\rm ad}}

\newcommand{\GL}{{\rm \bf GL}}

\newcommand{\MT}{{\rm \bf MT}}

\newcommand{\alg}{\textnormal{alg}}

\newcommand{\an}{\textnormal{an}}

\newcommand{\bH}{{\mathbf H}}

\def\Fp{\mathfrak{p}}
\def\FS{\mathfrak{S}}

\def\FH{\mathfrak{H}}

\def\Fs{\mathfrak{s}}
\def\Fg{\mathfrak{g}}
\def\Fh{\mathfrak{h}}

\def\Fn{\mathfrak{n}}

\newcommand{\cF}{\mathcal{F}}
\newcommand{\cA}{{\mathcal A}}

\newcommand{\cD}{{\mathcal D}}

\newcommand{\cR}{{\mathcal R}}
\newcommand{\cS}{{\mathcal S}}

\newcommand{\bP}{\mathbf P}
\newcommand{\bQ}{\mathbf Q}
\newcommand{\bL}{\mathbf L}
\newcommand{\bN}{\mathbf N}
\newcommand{\bM}{\mathbf M}
\newcommand{\bS}{\mathbf S}
\newcommand{\bz}{\mathbf z}
\newcommand{\br}{\mathbf r}
\newcommand{\bx}{\mathbf x}
\newcommand{\by}{\mathbf y}
\newcommand{\bn}{\mathbf n}
\newcommand{\bj}{\mathbf j}
\newcommand{\bss}{\mathbf s}

\newcommand{\oQ}{\overline{\QQ}}

\newcommand{\SL}{{\mathbf{SL}}}

\newcommand{\Hod}{\textnormal{Hod}}
\newcommand{\Reel}{\textnormal{Re}\,}
\newcommand{\Ima}{\textnormal{Im}\,}
\newcommand{\Ad}{\textnormal{Ad}\,}

\begin{document}
\title{Tame topology of arithmetic quotients and algebraicity of Hodge
loci}
\author{B. Klingler}
\thanks{This work was partially supported by an Einstein Foundation's
  professorship}

\begin{abstract}
We prove that the uniformizing map of any arithmetic quotient, as well
as the period map associated to any pure
polarized $\ZZ$-variation of Hodge structure $\VV$ on a smooth complex
quasi-projective variety $S$, are topologically
tame. As an easy corollary of these results and of Peterzil-Starchenko's
o-minimal GAGA theorem we obtain that the Hodge locus of
$(S, \VV)$ is a countable union of algebraic subvarieties of $S$ (a
result originally due to Cattani-Deligne-Kaplan).
\end{abstract}

\maketitle

\section{Introduction.}

\subsection{Arithmetic quotients}
Arithmetic quotients are real analytic manifolds of the form
$S_{\Gamma,G, M}:= \Gamma \backslash G/M$, for $\G$ a connected
semi-simple linear algebraic $\QQ$-group, $G:= \G(\RR)^+$ the real Lie group
connected component of the identity of $\G(\RR)$, $M \subset G$ a
connected compact subgroup and $\Gamma \subset \G(\QQ)^+:= \G(\QQ) \cap G$
a neat arithmetic group. By a morphism of arithmetic quotients
we mean the real analytic map $f: S_{\Gamma', G', M'} \lo S_{\Gamma,G, M}$
deduced from a morphism $f: \G' \lo \G$ of semi-simple linear algebraic
$\QQ$-group such that $f(M') \subset M$ and $f(\Ga') \subset \Ga$.

Such quotients are ubiquitous in various parts of
mathematics. For $M= \{1\}$ the arithmetic quotients $
S_{\Gamma, G, \{1\}}= \Gamma \backslash G$ are the main players in homogeneous dynamics, for example Ratner's
theory \cite{Rat0}, \cite{Rat1}. For $K\subset G$ a maximal compact subgroup the
arithmetic quotients $S_{\Gamma, G, K}$ are the arithmetic locally symmetric
spaces, for instance the arithmetic hyperbolic manifolds $\Gamma
\backslash SO(n,1)^+/SO(n)$. They are intensively studied by differential geometers and group
theorists. When $G$ is moreover of Hermitian type then $S_{\Gamma, G,K}$ is a so-called 
arithmetic variety (a connected Shimura variety if $\Gamma$ is a
congruence subgroup), a smooth complex quasi-projective variety
(naturally defined over $\oQ$ in the Shimura case). The simplest
examples of connected Shimura varieties are the modular
curves $\Gamma_0(N) \backslash SL(2, \RR)/ SO(2)$. Shimura varieties play a paramount role in
arithmetic geometry and the Langlands program. 

\subsection{Moderate geometry of arithmetic
  quotients} \label{tameness}

For $S_{\Ga, G, K}$ a connected Shimura variety, the study of the topological 
tameness properties of the uniformization map $\pi: G/M \lo S_{\Ga, G,
K}$ recently provided a crucial tool for solving longstanding
algebraic and arithmetic questions  (see \cite{Pil}, \cite{PT},
\cite{KUY}, \cite{Tsi}, \cite{KUY18},
\cite{MPT}). Here tameness is understood in the sense 
proposed by Grothendieck \cite[\textsection 5 and 6]{Gro} and
developed by model theory under the name ``o-minimal structure'' (see below). The first goal of this paper is to develop a similar
study for a general arithmetic quotient $S_{\Ga, G, M}$. In particular {\em from now on we
assume in this section that $\Ga$ is not cocompact in $G$} (the results are trivially
true if $\Ga$ is cocompact).

Among real analytic manifolds the ones with the tamest geometry
at infinity are certainly the complex algebraic ones, as they are
simply described as the zero locus of finitely many polynomial
equations with complex coefficients. However most arithmetic quotients have no complex algebraic 
structures, as they do not even admit a complex analytic one (for
instance for obvious dimensional reasons). What about a semi-algebraic
structure? One can show that any
arithmetic quotient $S_{\Ga, G, M}$ is real-analytically equivalent to a non-singular semi-algebraic set. On the other hand such 
abstract semi-algebraic models are useless if they don't satisfy some basic
functorial properties. A crucial feature of the geometry of arithmetic quotients is the
existence of infinitely many real-analytic finite self-correspondences: any
element $g \in \G(\QQ)$ commensurates $\Gamma$ (meaning that the
intersection $g\Gamma g^{-1} \cap \Gamma$ is of finite index in both
$\Gamma$ and $g\Gamma g^{-1}$) hence defines a Hecke correspondence 
\begin{equation} \label{Hecke}
\xymatrix{
 c_g = (c_1, c_2) :& S_{\Ga, G, M} &S_{g^{-1}\Gamma g \cap \Gamma, G, M} \ar[l]^{c_1} 
\ar[r]^{g \cdot} \ar@/_1pc/[rr]_{c_2} & S_{\Gamma \cap g \Gamma g^{-1}, G, M}
\ar[r] &S_{\Ga, G, M}}\;\;,
\end{equation}
where the extreme maps are the natural finite \'etale projections. We
would like these Hecke correspondences to
be real algebraic. Such functorial real algebraic models do exist in certain cases: see
\cite{Jaf75}, \cite{Jaf78}, \cite{Le79}; but I don't know of any general procedure for producing
such a nice semi-algebraic structure on all arithmetic quotients. Hence our
need to work in a more general notion of tame geometry.

Recall that a structure in the sense of
model theory is a collection $\cS= (S_n)_{n \in \NN^*}$,
where $S_n$ is a set of subsets of $\RR^n$ (called the {\em $\cS$-definable
sets}), such that: all algebraic subsets of $\RR^n$ are in $S_n$; $S_n$ is  a boolean subalgebra of the power set of $\RR^n$;
if $A\in S_n$ and $B \in S_m$ then $A \times B \in
S_{n+m}$; if $p: \RR^{n+1} \lo \RR^n$ is a linear projection and $A
\in S_{n+1}$ then $p(A) \in S_n$. A function $f: \RR^n \lo \RR^m$ is
said to be $\cS$-definable if its graph 
is $\cS$-definable. A structure $\cS$ is said to be {\em o-minimal} if the definable subsets of
$\RR$ are precisely the finite unions of points and intervals
(i.e. the semi-algebraic subsets of $\RR)$. This o-minimal axiom guarantees the possibility of doing
geometry using definable sets as basic blocks: it excludes infinite
countable sets, like $\ZZ \subset \RR$, as well as Cantor sets or
space-filling curves, to be
definable. Intuitively, subsets of
$\RR^n$ definable in an o-minimal structure are the ones having at
the same time a reasonable local topology and a tame topology at
infinity. Given an o-minimal structure $\cS$, there is an obvious
notion of $\cS$-definable manifold: this is a manifold $S$ admitting a {\em
  finite} atlas of charts $\varphi_i: U_i \lo \RR^n$, $i \in I$, such
that the intersections $\varphi_i(U_i \cap U_j)$, $i, j \in I$, are
$\cS$-definable subset of $\RR^n$ and the change of coordinates
$\varphi_i \circ \varphi_j^{-1}: \varphi_j(U_i \cap U_j) \lo
\varphi_i(U_i \cap U_j)$ are $\cS$-definable maps. 

The simplest o-minimal structure is $\RR_{\alg}$, the definable sets being the
semi-algebraic subsets. Fortunately there exist more general o-minimal structures. A result of Van den
Dries based on Gabrielov's results \cite{Gabrielov} shows that the
structure $$\RR_\an:=\langle \RR,\, +, \,\times,\, <, \{f\} \, 
\textnormal{for} \, f \,\textnormal{restricted analytic function}
\rangle$$ generated from $\RR_\alg$ by adding the restricted analytic
functions is o-minimal. Here a real function on $\RR^n$ is restricted analytic
if it is zero outside $[0,1]^n$ and coincides on $[0,1]^n$ with a real analytic function
$g$ defined on a neighbourhood of $[0,1]^n$. The $\RR_\an$-definable
sets of $\RR^n$ are the globally subanalytic subsets of $\RR^n$
(i.e. the ones which are subanalytic in the compactification $\bP^n
\RR$ of $\RR^n$). A deep result of Wilkie
\cite{Wil} states that the structure $\RR_{\exp}:= \langle \RR, +,
\times, <, \exp: \RR \lo \RR \rangle$ generated from $\RR_{\alg}$ by making the real
exponential function definable is also o-minimal. Finally the
structure $\RR_{\an, \exp}:= \langle \RR,\, +, \,\times,\, <,
\,\exp, \{f\} \, \textnormal{for} \,  f \,\textnormal{restricted
  analytic function} \rangle$ 
generated by $\RR_{\an}$ and $\RR_{\exp}$ is still o-minimal \cite{VdM}.

The first result of this paper is the following:
\begin{theor}
\label{definability}
Let $\G$ be a connected linear semi-simple algebraic $\QQ$-group,
$\Gamma \subset \G(\QQ)^+$  a torsion-free arithmetic lattice 
 of $G:= \G(\RR)^+$, and $M \subset G$ a connected compact subgroup.
\begin{itemize}
\item[(1)]
The quotient $S_{\Gamma, G, M} := \Gamma \backslash G/M$ has a natural structure of
$\RR_{\an}$-definable manifold, induced from the real-analytic
structure with corners of its Borel-Serre compactification
$\overline{S_{\Ga, G, M}}^{BS}$. 
\item[(2)] \label{unif} Let $G/M$ be endowed with its natural semi-algebraic
  structure (see \Cref{G/M semi}) and $\FS \subset G/M$ be a
  semi-algebraic Siegel set (see \Cref{Siegel sets} for the definition
  of Siegel sets). Then 
$$\pi_{|\FS}: \FS  \lo  S_{\Gamma,G, M} $$ is $\RR_{\an}$-definable (where $S_{\Gamma,G, M}$ is endowed
with its $\RR_\an$-structure defined in $(1)$). 

In particular, there exists a semi-algebraic fundamental set $\FF \subset G/M$ for the
  action of $\Gamma$ on $G/M$ such that 
$$\pi_{|\FF}: \FF \lo S_{\Gamma,G, M} $$ is $\RR_{\an}$-definable.

\item[(3)] Any morphism $f: S_{\Ga', G', M'} \lo S_{\Ga, G, M}$ of
  arithmetic quotients is $\RR_\an$-definable. In particular the Hecke
  correspondences $c_g$, $g \in \G(\QQ)^+$, on $S_{\Ga, G, M}$ are $\RR_{\an}$-definable. 
\end{itemize}
\end{theor}

\begin{rems}
\begin{itemize}
\item[(1)] \Cref{definability}(1) is a corollary of general facts
  about compact real analytic manifolds with corners and Borel-Serre
  theory: see \Cref{general}. 
\item[(2)] Siegel sets are the main tool in the understanding of the
  Borel-Serre compactification $\overline{S_{\Ga, G, M}}^{BS}$, and
  \Cref{definability}(2) is the crucial ingredient for studying
  arithmetic quotients via o-minimal techniques.
\item[(3)] Borel-Serre compactification is not functorial: a morphism $f: S_{\Ga', G', M'} \lo S_{\Ga, G, M}$ of
  arithmetic quotients does not in general extend to a morphism of real analytic
  manifolds with corners between $\ol{S_{\Ga', G', M'}}^{BS}$ and
  $\ol{S_{\Ga, G, M}}^{BS}$. Hence \Cref{definability}(3) is
  non-trivial but follows from a finiteness result for Siegel sets due
  to Orr (see \Cref{orr}).
\end{itemize}
\end{rems}

When $S_{\Gamma, G, K}$ is an arithmetic variety, the
main result of \cite{PetStar} (for $S_{\Gamma, G,
  K} = \cA_g$ the moduli space of principally polarized Abelian
varieties of dimension $g$) and \cite[Theor.1.9]{KUY} (in general)
states that the map $\pi_{|\FF}: \FF \lo S_{\Gamma,G, K}$ is $\RR_{\an, \exp}$-definable
when the arithmetic variety $S_{\Gamma, G, K}$ is endowed with the
$\RR_{\an}$-definable manifold structure deduced from its Baily-Borel
compactification. \Cref{definability}(iii) can be thought as a generalisation
of this result, as one can prove that the two
$\RR_\an$-definable structures on $S_{\Gamma, G, K}$ given by the
Borel-Serre compactification on the one hand and the Baily-Borel
compactification on the other hand do not coincide but do define the
same $\RR_{\an, \exp}$-definable structure. As we won't need this
result in this paper we just provide the simplest illustration:

\begin{Example} \label{Ex}
Let $\FH$ be the Poincar\'e upper half-plane and $Y_0(1)$ the modular
curve $\SL(2, \ZZ) \backslash \FH$. A semi-algebraic fundamental domain for the
action of $\SL(2, \ZZ)$ on $\FH$ is given by $$\cF:=\{ (x, y) \in \FH\; | \; x^2 + y^2 \geq 1, -1/2 < x< 1/2\}\;\;.$$
The Borel-Serre compactification $\ol{Y_0(1)}^{BS}$ is obtained by adding a circle at
infinity to $Y_0(1)$, corresponding to the compactification $\ol{\cF}$
of $\cF$ obtained by glueing the segment $\{ y= \infty, -1/2 < x<
1/2\}$ to $\cF$. The Baily-Borel compactification $X_0(1):= \ol{Y_0(1)}^{BB}$
is the one-point compactification of $Y_0(1)$ and is naturally
identified with the complex projective line $\bP^1\CC$. The natural map
$\ol{Y_0(1)}^{BS} \lo \ol{Y_0(1)}^{BB}$ contracting the circle at
infinity to a point sends a point $(x, t= 1/y) \in [-1/2, 1/2] \times [0, 1)$
close to the circle at infinity $t=0$ to the point $[1, z= \exp(2 \pi i x) \exp (-2
\pi/t)] \in \bP^1\CC$. This map is not globally subanalytic but it is
$\RR_{\an, \exp}$-definable.
\end{Example}

It is worth noticing that the proof of the general
\Cref{definability}(iii) is much easier than the one in \cite{PetStar} (which uses explicit theta
functions) or the one in \cite{KUY} (which relies in a fundamental
way on the delicate toroidal compactifications \cite{AMRT}). Indeed
these proofs, which apply only to arithmetic varieties, 
insist on using complex analytic maps, which obscure to some
extent the o-minimality issues.

\subsection{Moderate geometry of period mappings} Arithmetic quotients
of interest to the algebraic geometers arise as quotients of period
domains 
(or more generally Mumford-Tate domains) in Hodge theory. Let $S$ be a
smooth complex quasi-projective variety and $\VV \lo S$ a polarized
variation of $\ZZ$-Hodge structures (PVHS) of weight $k$ on $S$. A typical example of such a
PVHS is $\VV = R^kf_*\ZZ$ for $f: \mathcal{X} \lo
S$ a smooth proper morphism; in which case we say that $\VV$ is geometric.
We refer to \cite{K17} and references therein for the
relevant background in Hodge theory, which we use thereafter. 
Let $\MT(\VV)$ be the generic Mumford-Tate group
associated to $\VV$ (this is a connected reductive $\QQ$-group) and $\G$
its associated adjoint semi-simple $\QQ$-group. The group $G:= \G(\RR)^+$ acts by
holomorphic transformations and transitively on the Mumford-Tate
domain $\cD=G/M$ associated to $\MT(\VV)$, with compact isotropy denoted
by $M$. If $\Ga$ is a torsion free arithmetic quotient of $G$ the
quotient $S_{\Ga, G, M}$ is a complex analytic manifold (which carries
an algebraic structure in only very few cases). Replacing if necessary $S$ by a finite \'etale cover, the PVHS
$\VV$ on $S$ is completely described by its holomorphic period map
$\Phi_S: S \lo \Hod^0(S, \VV)$, where the connected Hodge variety $\Hod^0(S,
\VV)$ associated to the pair $(S, \VV)$ is an arithmetic quotient $S_{\Ga, G, M}$ for a suitable
torsion-free arithmetic subgroup $\Gamma \subset G$, admitting a
natural complex analytic structure but usually no complex algebraic structure (notice that here and in the
rest of the text we use the same symbol $S$ for denoting a complex
algebraic variety and its associated complex analytic space).

We prove that period maps have a moderate
geometry:

\begin{theor} \label{definability period map}
Let $\VV \rightarrow S$ be a polarized variation of pure Hodge
structures of weight $k$ over a smooth complex quasi-projective variety $S$.
Let $\Phi_S: S \lo \Hod^0(S, \VV):= S_{\Ga, G, M}$ be the horizontal
holomorphic period map associated to $\VV$. 

Then $\Phi_S$ is $\RR_{\an, \exp}$-definable (where we endow $S$ with its canonical
$\RR_{\alg}$-definable structure coming from its algebraic structure and $S_{\Ga, G, M}$ with its $\RR_{\an}$-definable
structure defined in \Cref{definability}(1)). 
\end{theor}

\begin{rems}
\begin{itemize}
\item[(1)] Notice that \Cref{definability period map} is easy in the
  rare case when the connected
  Hodge variety $S_{\Ga, G, M}$ is compact. In that case, consider $\ol{S}$ a
  smooth projective compactification of $S$ with normal crossing divisor at
infinity. It follows from Borel's monodromy theorem \cite[Lemma
(4.5)]{Schmid} and the fact that the cocompact lattice $\Ga$ does not
contain any unipotent element \cite[Cor. 11.13]{Rag72} that the
monodromy at infinity of $\VV$ is finite. Thus, replacing if
necessary $S$ by a finite \'etale cover, the PVHS $\VV$ extends to
$\ol{S}$. Equivalently the period map
$\Phi_S: S \lo \Hod^0(S, \VV):= S_{\Ga, G, M}$ extends to a
period map $\Phi_{\ol{S}}: \ol{S} \lo S_{\Ga, G, M}$. In particular
the period map 
$\Phi$ is definable in $\RR_\an$ in that case.

\item[(2)] When the connected Hodge variety $S_{\Ga, G, M}$ is a
  connected Shimura variety, \Cref{definability period map} implies
  that $\Phi_S: S \lo S_{\Ga, G, M}$ is an algebraic map (see \Cref{algebraic}), a classical
  result due to Borel \cite[Theor. 3.10]{Bor72}. Hence \Cref{definability period map}
  can be thought as an extension of Borel's result to the general
  case where the connected Hodge variety $S_{\Ga, G, M}$ has no
  algebraic structure. On the other hand, notice that Borel
  \cite[Theor.A]{Bor72} proves in the Shimura case the
  stronger result that $\Phi_S$ extends to a holomorphic map $\Phi_{\ol{S}}: \ol{S} \lo \ol{S_{\Ga, G,
      M}}^{BB}$ (where, as above, $\ol{S_{\Ga, G,
      M}}^{BB}$ denotes the projective Baily-Borel compactification of the
  Shimura variety $S_{\Ga, G,M}$). 

\item[(3)] A long standing conjecture of Griffiths, whose proof has recently been
announced in \cite{GGLR}, states that $\Phi(S)$ admits a natural
completion $\ol{\Phi(S)}$ as a projective algebraic variety (see
\cite{Sommese} for earlier results in this direction) and that the map
$\Phi_S: S \lo \Phi(S)$ extends to an algebraic map $\Phi_{\ol{S}}:
\ol{S} \lo \ol{\Phi(S)}$. This implies that $\Phi(S)$ has a tame
topology, but says nothing about the tameness of $\Phi_S: S \lo
S_{\Ga, G, M}$, as the relation
between the projective compactification $\ol{\Phi(S)}$ of \cite{GGLR} and the Hodge
variety $S_{\Ga, G,M}$ is far from clear. Hence \Cref{definability
  period map} seems to go in a direction different from the one
followed in \cite{GGLR}.
\end{itemize}
\end{rems}

The main ingredient in the proof of \Cref{definability period map} is
the following finiteness result on the geometry of Siegel sets (we
refer to \Cref{Siegel sets} for the precise
definition of Siegel sets):

\begin{theor} \label{finitely many Siegel} 
Let $\Phi: (\Delta^*)^n \lo S_{\Ga, G, M}$ be a period
map with unipotent monodromy on a product of punctured disks. Let $\tilde{\Phi}: \FH^n 
\lo G/M$ be its lifting to the universal cover $\FH^n$ of $(\Delta^*)^n$. 

For any given constants $R>0$ and $\eta >0$ there exists finitely
many Siegel sets $\FS_i \subset G/M$, $i \in I$, such that
$\tilde{\Phi}(\bz) \in \bigcup_{i \in I}\FS_i$
whenever $|\Reel \bz| \leq C$
and $\Ima \bz \geq \eta$ (where
$|\Reel \bz| := \inf_{1\leq j \leq n} |\Reel z_i|$ and $\Ima \bz := \inf_{1 \leq j \leq n} \Ima z_i$).
\end{theor}

In the one-variable case ($n=1$) \Cref{finitely many Siegel} is due to Schmid (see
\cite[Cor. 5.29]{Schmid}), with $|I|=1$. In the multivariable case, Green, Griffiths, Laza and
Robles \cite[Claims A.5.8 and A.5.9]{GGLR} show that the result with
$|I|=1$ does not hold. As in the one-variable case, our proof of \Cref{finitely
  many Siegel} is a corollary of the (multivariable) ${SL_2}^n$-orbit
theorem of Cattani-Kaplan-Schmid \cite[(4.20)]{CKS}. 

It seems it might be possible to use only the simpler $SL_2$-orbit
theorem in one variable to prove \Cref{definability period map}, we hope to come back
to this question in a subsequent work.

\subsection{Algebraicity of Hodge loci}
Recall that the Hodge locus $\HL(S, \VV)
\subset S$ associated to the PVHS $\VV$ is the set of 
points $s$ in $S$ for which exceptional Hodge tensors for $\VV_{s}$ do
occur. The locus $\HL(S, \VV)$ is easily seen to be a countable union
of irreducible complex analytic subvarieties of $S$, called special
subvarieties of $S$ associated to $\VV$. If $\VV = R^kf_*\QQ$ for $f: \mathcal{X}
\lo S$ a smooth proper morphism, it follows from the Hodge
conjecture that the exceptional Hodge tensors in $\VV_{s}$ come from exceptional
algebraic cycles in some product $\mathcal{X}_s^N$. A Baire category
type argument then implies that every special subvariety of $S$ ought to be algebraic. As a
corollary of \Cref{definability}, \Cref{definability 
  period map}, and Peterzil-Starchenko's GAGA o-minimal \Cref{PS} we obtain an alternative proof of this result originally proven by 
Cattani, Deligne and Kaplan \cite{CDK95}: 

\begin{theor}  \label{algebraicity}
The special subvarieties of $S$ associated to $\VV$ are algebraic, i.e.
the Hodge locus $\HL(S, \VV)$ is a countable union of closed irreducible algebraic
subvarieties of $S$.
\end{theor}

The proof of \Cref{algebraicity} in
\cite{CDK95} works as follows. Let $\ol{S}$ be a smooth
compactification of $S$ with a simple normal crossing divisor $D$ at
infinity. Locally in the analytic topology
$S$ identifies with $(\Delta^*)^r \times
\Delta^l$ inside $\ol{S} = \Delta^{r+l}$ (where $\Delta$ denotes the
unit disk). The ${SL_2}^n$-orbit theorems of \cite{Schmid} and 
\cite{CKS} describes extremely precisely the asymptotic of the period
map $\Phi_S$ on $(\Delta^*)^r \times
\Delta^l$. Using this description, Cattani, Deligne and Kaplan manage
to write sufficiently explicitly the equation of the locus $S(v)
\subset (\Delta^*)^r \times \Delta^l$ of the points at which some
determination of a given multivalued flat section $v$ of $\VV$ is a Hodge
class to prove that its closure $\ol{S(v)}$ in $\Delta^{r+l}$ is
analytic in this polydisk. Our proof via \Cref{definability period
  map} bypass these delicate local computations, hence seems a worthwhile
simplification.

In view of \Cref{definability period map}, its corollary
\Cref{algebraicity}, and the recent proof \cite{BaT} (using
\Cref{algebraicity} and o-minimal technics) of
the Ax-Schanuel conjecture for pure Hodge
varieties stated in \cite[Conj. 7.5]{K17}, we hope to convey the idea
that o-minimal geometry is an important tool in variational Hodge theory. We refer to \cite[section 1.5]{K17} for
possible applications of these results to the structure of $\HL(S,
\VV)$.

\Cref{algebraicity} has been extended to the case of (graded
polarizable, admissible) variation of mixed Hodge structures in \cite{BP1}, \cite{BP2},
\cite{BP3}, \cite{BPS}, \cite{KNU11} using \cite{CDK95} and the ${SL_2}^n$-orbit
theorem of \cite{KNU08} which  extends \cite{Schmid} and
\cite{CKS} to the mixed case. Our o-minimal proof of \Cref{definability period
  map} certainly extends to this case, thus giving a simpler proof of the
algebraicity of Hodge loci in full generality. We will come back to
this problem is a sequel to this paper.

\subsection{Acknowledgments} I would like to thank Patrick Brosnan,
who asked me some time ago about a proof of \Cref{algebraicity} using o-minimal technics; Mark Goresky,
Lizhen Ji and Arvind Nair, who made me notice that the map
from the Borel-Serre compactification to the Baily-Borel one is not
subanalytic, and that Borel-Serre compactifications are not
functorial; Wilfried Schmid, who confirmed to me that \Cref{finitely many
  Siegel} should hold; and especially Colleen Robles, for our fruitful
exchanges about \Cref{finitely many Siegel}.

\subsection{}After finishing this paper, I learned from Phillip
Griffiths and Akshay Venkatesh that Benjamin
  Bakker and Jacob Tsimerman are announcing results similar to
  \Cref{definability period map} and \Cref{finitely many Siegel}.

\section{Proof of \Cref{definability}}

\subsection{Semi-algebraic structure on $G/M$}
The existence of a natural semi-algebraic structure on $G/M$, which we
use in \Cref{definability}(iii) is classical. We provide a proof for the
convenience of the reader.

\begin{lem} \label{G/M semi}
Let $\G$ be a connected semi-simple linear algebraic $\QQ$-group, $G:= \G(\RR)^+$ the real Lie group
connected component of the identity of $\G(\RR)$, and $M \subset G$ a
connected compact subgroup. Then $G/M$ admits a natural structure of a
semi-algebraic set, and the projection map $G \lo G/M$ is semi-algebraic.
\end{lem}

\begin{rem} \label{rem1}
In general $G/M$ does not admit a structure of real algebraic variety.
This is already true for $G$: for instance the group $SO(p,q)$ is a real
algebraic variety but its connected component $G:= SO(p,q)^+$
is only semi-algebraic for $p\geq q>0$. On the other hand any compact real Lie group
$M$ admits a natural structure $\bM_\RR$ of real algebraic group, see \cite[Th. 5, p.133]{OV}.
\end{rem}

\begin{proof}
Let $\bH_\RR$ be a real reductive algebraic subgroup of $\G_\RR$. Recall the classical:
\begin{sublem} (Chevalley)
There exists a finite dimensional $\G_\RR$-module $W$ and a line $l \subset W$ such
that the stabilizer in $\G_\RR$ of $l$ is precisely $\bH_\RR$.
\end{sublem}
\begin{proof}
Consider the action of $\G_\RR$ on itself by left multiplication. Then
the stabilizer of the closed subvariety $\bH_\RR$ is $\bH_\RR$
itself. Thus, $\bH_\RR$ is also the stabilizer of the ideal
$I(\bH_\RR) \subset \RR[\G_\RR]$. Let us choose a finite-dimensional
space $I \subset I(\bH_\RR)$ which generates that ideal. As
$I(\bH_\RR)$ is a rational $\bH_\RR$-module we can also assume that $I$ is
$\bH_\RR$-stable. Hence it is contained in a finite dimensional
$\G_\RR$-submodule $J \subset \RR[\G_\RR]$. As $\bH_\RR$ is the
stabilizer of $I$, it is also the stabilizer of the line $l:=\bigwedge^n
I$ of the $\G_\RR$-module $W:= \bigwedge^nJ$, where $n := \dim(I)$.
\end{proof}

Apply the previous result to $\bH_\RR = \bM_\RR$, the real algebraic
subgroup of $\G_\RR$ such that $\bM_\RR(\RR)=M$. As the group $M$ is
compact connected, it not only stabilizes the line $l$ but fixes any
generator $v$ of $l$. 

By a classical result of Hilbert (see \cite[Ch. VIII, \textsection
14]{Weyl}) the (graded) algebra $\RR[W]^{\bM_\RR}$ of
$\bM_\RR$-invariant polynomials on $W$ is finitely generated, say by
homogeneous elements $p_1, \cdots, p_d$. Consider the real algebraic map
$p : \G(\RR) \rightarrow W \rightarrow \RR^d$
obtained by composing the orbit map of the vector $v \in W$ with $
(p_1, \cdots, p_d): W \lo \RR^d$. It identifies $\G(\RR)/M$ with the image
$p(\G(\RR))$, hence $G/M$ with a connected component of
$p(\G(\RR))$. As $p$ is real algebraic the subset $p(\G(\RR))$, hence
its connected component $G/M$, is semi-algebraic.
As $p$ is real-algebraic, the projection $G \lo G/M$ is semi-algebraic.
\end{proof}

\subsection{Definability of arithmetic quotients: proof of
  \Cref{definability}(1)} \label{general}

In this section we recall basic facts about real manifolds with
corners and their definable version. The proof of \Cref{definability}(i)
is then a direct consequence of the existence of the Borel-Serre
compactification $\ol{S_{\Ga, G, M}}$, see \cite{BS}, \cite{BJ}.

\subsubsection{Real analytic manifolds with corners}
From the analytic point of view, the class of real analytic manifolds
with corners is natural: a compact real-analytic manifold with corners
is the real version of the compactification of a complex analytic manifold
by a normal crossing divisor. However this class of
manifolds has been poorly studied and even their definition is not 
universally agreed. We use the one given by \cite{Dou}, which has been
clarified and developed in \cite{Joyce}. For the convenience of the reader we recall the basic
definitions but we refer to \cite{Joyce} for more details. Notice
that Joyce works in the $C^\infty$ context, but all the definitions we
need translate literally to the real-analytic setting by replacing ``smooth'' with ``real-analytic''.

Let $X$ be a paracompact Hausdorff topological space $X$ and $n\geq 1$
an integer. An $n$-dimensional chart with corners on $X$ is a pair 
$(U, \varphi)$ where $U$ is an open subset in $\RR^n_k:= \RR_{\geq 0}^k
\times \RR^{n-k}$ for some $0 \leq k \leq n$ and $\varphi: U \lo X$ is a homeomorphism with a
non-empty open set $\varphi(U)$.

Given $A \subset \RR^m$ and $B \subset \RR^n$ and $\alpha:A \lo B$
continuous, we say that $\alpha$ is real-analytic if it extends to a
real-analytic map between open neighborhoods of $A$, $B$.

Two $n$-dimensional charts with
corners $(U, \varphi)$, $(V, \psi)$ on $X$ are said real-analytically compatible if $\psi^{-1}
\circ \varphi: \varphi^{-1}(\varphi(U) \cap \psi(V)) \lo
\psi^{-1}(\varphi(U) \cap \psi(V))$ is a homeomorphism and $\psi^{-1}
\circ \varphi$ (resp. its inverse) are real-analytic in the sense
above.

An $n$-dimensional real analytic atlas with corners for $X$ is a system $\{(U_i,
\varphi_i): i \in I\}$ of pairwise real-analytically compatible charts
with corners on $X$ with $X= \bigcup_{i\in I} \varphi_i(U_i)$. We call
such an atlas maximal if is not a proper subset of any other
atlas. Any atlas is contained in a unique maximal atlas: the set of
all charts with corners $(U, \varphi)$ on $X$ compatible with $(U_i,
\varphi_i)$ for all $i \in I$. 

A real-analytic manifold with corners of dimension $n$ is a paracompact Hausdorff
topological $X$ equipped with a maximal $n$-dimensional real-analytic
atlas with corners. Weakly real-analytic maps between real-analytic
manifolds with corners are the continuous maps which are real-analytic
in charts (cf. \cite[def. 3.1]{Joyce}, where a stronger notion of
real-analytic map is also defined; we won't need this strengthened notion).

Given $X$ a real-analytic $n$-manifold with corners, one defines its
boundary $\partial X$ (cf. \cite[def. 2.6]{Joyce}. This is a
real-analytic $n$-manifold with corners for $n>0$, endowed with an
immersion (not necessarily injective) $i_X: \partial X
\lo X$ (cf. \cite[prop.2.7]{Joyce}) which is real-analytic
(\cite[Theor. 3.4.(iv)]{Joyce}) in particular weakly real-analytic.

\subsubsection{$\cR$-definable manifolds with corners}
Let $\cR$ be any fixed o-minimal expansion of $\RR$. The notion of
$\cR$-definable manifold is given in \cite[chap.10]{VDD} and in
\cite[p.507]{VdM2}. From now on $\cR$ denotes an o-minimal expansion
of $\RR_\an$. We will need the extended notion of
$\cR$-definable manifold with corners.

Let $X$ be a paracompact Hausdorff
topological space $X$. An $n$-dimensional chart with
corners $(U, \varphi)$ on $X$ is said to be $\cR$-definable if $U$ is an $\cR$-definable subset of
$\RR^n$ (equivalently: of $\RR^n_k$).

Two $n$-dimensional $\cR$-definable charts with
corners $(U, \varphi)$, $(V, \psi)$ on $X$ are said $\cR$-compatible if $\psi^{-1}
\circ \varphi: \varphi^{-1}(\varphi(U) \cap \psi(V)) \lo
\psi^{-1}(\varphi(U) \cap \psi(V))$ is an $\cR$-definable
homeomorphism between $\cR$-definable subsets 
$\varphi^{-1}(\varphi(U) \cap \psi(V))$ and $\psi^{-1}(\varphi(U) \cap
\psi(V))$ of $\RR^n$. 

An $n$-dimensional $\cR$-definable atlas with corners for $X$ is a system $\{(U_i,
\varphi_i): i \in I\}$, {\em $I$ finite}, of pairwise $\cR$-compatible $\cR$-definable charts
with corners on $X$ with $X= \cup_{i\in I} \varphi_i(U_i)$. Two such
atlases $\{(U_i, \varphi_i): i \in I\}$ and$\{(V_j,
\psi_j): j \in J\}$  are said $\cR$-equivalent if all the ``mixed''
transition maps $\psi_j \circ \varphi_i^{-1}$ are $\cR$-definable.

An $\cR$-definable manifold with corners of dimension $n$ is a paracompact Hausdorff
topological $X$ equipped with an $\cR$-equivalence class of $n$-dimensional $\cR$-definable
atlas with corners.

\begin{rem}
Notice that the definitions of real-analytic manifold with corners
and $\cR$-definable manifold with corners are parallel,
except the crucial fact that we work in a strictly finite setting for
$\cR$-definable manifolds: the set $I$ of charts has to be finite. This finiteness
condition, in addition to the definability condition, ensures the
tameness at infinity of the $\cR$-definable manifolds with corners.
\end{rem}

We say that a subset $Z\subset X$ is $\cR$-definable (resp. open or closed)
if $\varphi_i^{-1}(Z \cap \varphi_i(U_i))$ is an $\cR$-definable (resp. open or closed) subset
of $U_i$ for all $i \in I$. An $\cR$-definable map between 
$\cR$-definable manifolds (with corners) is a map whose graph is an $\cR$-definable subset of the
$\cR$-definable product manifold (with corners).

\subsubsection{Compact real-analytic manifolds with corners are
$\RR_{\an}$-definable}

\begin{prop} \label{ManifoldWC}
Let $X$ be a compact real-analytic $n$-manifold with corners. Then $X$ has
a natural structure of $\RR_{\an}$-definable manifold with
corners. Moreover the map $i_X: \partial X \lo X$ is
$\RR_{an}$-definable. In particular the interior $X \setminus
i_X(\partial X)$ is an $\RR_{\an}$-definable manifold.
\end{prop}

\begin{proof}
For each point $x$ of $X$ choose $\varphi_x: U_x \lo (X, x)$ a real-analytic
chart with corners whose image $\varphi(U_x)$ is a neighborhood of
$x$. Without loss of generality we can assume that $U_x \subset \RR^n_k$ is relatively
compact and semi-analytic, hence $\RR_{\an}$-definable. Hence $(U_x,
\varphi_x)$ is a real-analytic chart with corners for $X$ which is
also an $\RR_\an$-definable chart with corners for $X$. 

Fix $x, y$ two points in $X$. The fact that the two real-analytic charts $(U_x, \varphi_x)$ and $(U_y,
\varphi_y)$ are real-analytically compatible implies immediately that
they are $\RR_{\an}$-compa\-tible.

The space $X$ is compact hence one can extract from the covering
family $\{(U_x, \varphi_x),\; x \in X\}$ a finite subfamily $\{ (U_i,
\varphi_i), \; i \in I\}$, such that $X = \bigcup_{i\in I}
\varphi_i(U_i)$: this is an $n$-dimensional
$\RR_{\an}$-definable atlas with corners for $X$, which defines a
structure of $\RR_\an$-definable manifold with corners on $X$.

One easily checks that this structure is independent of the choice of
the finite extraction $\{ (U_i,
\varphi_i), \;i \in I\}$ of $\{(U_x, \varphi_x), \;x \in X\}$, and
also of the choice of the relatively compact and semi-analytic subsets
$U_x$.

Hence $X$ has a natural structure of $\RR_{\an}$-definable manifold with
corners. The same procedure endows the compact real-analytic $(n-1)$-manifold
with corners $\partial X$ with a natural $\RR_\an$-definable
structure. The fact that $i_X: \partial X \lo X$ is weakly
real-analytic implies immediately that $i_X$ is $\RR_\an$-definable
and that the manifold $X \setminus i_X(\partial X)$ is $\RR_\an$-definable.
\end{proof}

\subsubsection{Proof of \Cref{definability}(1)}

\begin{proof}[\unskip\nopunct]
In \cite{BS} Borel and Serre construct a natural compactification
$\overline{S_{\Gamma, G, K}}^{BS}$ of any arithmetic locally symmetric space
$S_{\Gamma, G, K}$ in the category of real-analytic manifolds with
corners, using the notion of geodesic actions and $S$-spaces. 
In \cite[\textsection 3]{BJ} Borel and Ji give a uniform construction of a
compactification $\overline{S_{\Ga, G, M}}^{BS}$ of any arithmetic quotient $S_{\Ga, G, M}$
in the category of real-analytic manifolds with corners, simplifying
the approach of \cite{BS} as they do not rely anymore on the notion of
$S$-spaces and delicate inductions. 

The strategy of Borel-Ji consists in constructing a partial
compactification $\overline{G}^{BS}$ of $G$ in the category of
real-analytic manifolds with corners \cite[Prop.6.3]{BJ}, such that the left $\G(\QQ)^+$-action 
on $G$ (see \cite[prop. 3.12]{BJ}) and the commuting right $K$-action of a maximal
compact subgroup $K$ (see \cite[Prop.3.17]{BJ}) both extend to an action by weakly analytic maps to
$\overline{G}^{BS}$ (see proof of \cite[Prop. 6.4]{BJ}). For any neat
arithmetic subgroup $\Ga$ of $G$ and compact subgroup $M$ of $G$,  the action of $\Gamma \times M$ on
$\overline{G}^{BS}$ is free and proper. The quotient $\ol{S_{\Ga,
G,  M}}^{BS}:= \Gamma \backslash \ol{G}^{BS} /M$ provides a
compactification of the arithmetic quotient $S_{\Ga, G, M}$ in the
category of real-analytic manifolds with corners. Notice that the
compactification $\ol{S_{\Ga, G,  M}}^{BS}$ was already constructed by Kato and Usui \cite{KU}
when $S_{\Ga, G, M}$ is a Hodge variety (see \cite[Remark
3.18]{BJ}). More details on this construction will be given in the
next sections, as we will need them for proving
\Cref{definability}(ii) and (iii). 

For now, \Cref{definability}(1) follows immediately from the existence of $\ol{S_{\Ga,
G,  M}}^{BS}$ and \Cref{ManifoldWC}.

\end{proof}

\subsection{Siegel sets and definability of the uniformization map:
  proof of \Cref{definability}(2)}

\subsubsection{Siegel sets}  \label{Siegel sets}
A crucial ingredient in this paper in the classical notion of Siegel sets for
$\G$, which we recall now. We follow \cite[\textsection 2]{BJ} and refer to \cite[\S
12]{bor} for details.

We assume that the $\QQ$-rank of $\G$ is positive
(equivalently, $\Ga$ is not cocompact). Let $\bP$ be a $\QQ$-parabolic subgroup of $\G$. We denote by
$\bN_P$ ist unipotent radical and by $\bL_P$ the Levi quotient $\bN_P
\backslash \bP$ of $\PP$. Let $N_P$, $P$, and $L_P$ be the Lie groups
of real points of $\bN_P$, $\bP$ and $\bL_P$ respectively. Let $\bS_P$
be the split center of $\bL_{P}$ and $A_P$ the connected component of
the identity in $\bS_{P}(\RR)$. Let $\bM_P := \cap_{\chi \in
  X^*(\bL_{P})} \ker  \chi^2$ and $M_P= \bM_P(\RR)$. Then $L_P$ admits
a decomposition $L_P = A_P M_P$. 

Let $X$ be the symmetric space of maximal compact subgroups of $G :=\G(\RR)^+$. Choosing a point $x \in X$ corresponds to choosing a maximal
compact subgroup $K_x$ of $G$, or equivalently a Cartan involution
$\theta_x$ of $G$. The choice of $x$ defines a unique real Levi subgroup $\bL_{P,x} \subset \bP_\RR$
lifting $(\bL_P)_\RR$ which is $\theta_x$-invariant, see
\cite[1.9]{BS}. Although $\bL_P$ is defined over $\QQ$ this is not
necessarily the case for $\bL_{P,x} $. The parabolic group $P$
decomposes as 
\begin{equation} \label{LD}
P= N_P A_{P, x} M_{P_{x}} \;\;,
\end{equation}
inducing a horospherical decomposition of $G$:
\begin{equation} \label{HD}
G = N_P A_{P, x} M_{P_{x}} K_x \;\;.
\end{equation}
We recall (see \cite[Lemma 2.3]{BJ} that the right action of $P$ on
itself under the horospherical decomposition is given by
\begin{equation} \label{mult}
(n_0 a_0 m_o) (n, a, m) = (n_0 \cdot (a_0 m_0) n (a_0 m_0)^{-1}, a_0
a, m_0m) \;\;.
\end{equation}

{\it In the following the reference to the basepoint $x$ in various
subscripts is omitted.}
We let $\Phi(A_P, N_P)$ be the set of characters of $A_P$ on the Lie
algebra $\Fn_P$ of $N_P$, ``the roots of $P$ with respect to
$A_P$''. The value of $\alpha \in \Phi(A_P, N_P)$ on $a \in A_P$ is
denoted $a^\alpha$. Notice that the map $a \mapsto a^\alpha$ from
$A_P$ to $\RR^*$ is semi-algebraic.

There is a unique subset $\Delta(A_P, N_P)$ of
$\Phi(A_P, N_P)$ consisting of $\dim A_P$ linearly independent roots,
such that any element of $\Phi(A_P, N_P)$ is a linear combination with
positive integral coefficients of elements of $\Delta(A_P, N_P)$ to be
called the simple roots of $P$ with respect to $A_P$. 

\begin{defi}(Siegel set)
For any $t >0$, we define $A_{P, t} = \{ a \in A_P \; | \; a^\alpha
>t, \; \alpha \in \Delta(A_P, N_P) \}$.
For any bounded set $U \subset N_P$ and $W \subset M_{P}K$ the subset
$\FS:= U \times A_{P,t} \times W \subset G$ is called a
Siegel set for $G$ associated to $\bP$ and $x$.

Given $M \subset K$ a connected compact subgroup, a Siegel set $\FS$
for $G/M$ is the image of a Siegel set $\FS \subset G$ under the natural
projection map $G \lo G/M$.
\end{defi}

\begin{prop} \cite[Prop. 2.5]{BJ} \label{reduction}
\begin{itemize}
\item[(1)] There are only finitely many $\Gamma$-conjugacy classes of
  parabolic $\QQ$-subgroups. Let $\bP_1$, \dots, $\bP_k$ be a set of
  representatives of the $\Gamma$-conjugacy classes of parabolic
  $\QQ$-subgroups. There exists Siegel sets $\FS_i:= U_i \times A_{P_i,t_i}
  \times W_i$ associated to $\bP_i$ and $x_i$, $1 \leq i \leq k$, whose
  images in $\Gamma \backslash G$ cover the whole space. 
\item[(2)] For any two parabolic subgroups $\bP_i$ and Siegel
  sets $\FS_i$ associated to $\bP_i$, $i=1, 2$, the set
  $$\{ \gamma \in \Gamma \; |\; \gamma \FS_1 \cap \FS_{2} \not =
  \emptyset\}$$ is finite.
\item [(3)] Suppose that $\bP_1$ is not $\Gamma$-conjugate to
  $\bP_2$. Fix $U_i$, $W_i$, $i =1,2$. Then  $\gamma \FS_1\cap \FS_2=
  \emptyset$ for all $t_1, t_2$ sufficiently large.
\item[(4)] For any fixed $U, W$, when $t >>0$, $\gamma \FS \cap \FS=\emptyset$ for all
  $\gamma \in \Gamma - \Gamma_P$, where $\Gamma_P := \Gamma \cap P$.
\item[(5)] For any two different parabolic subgroups $\bP_1$ and
  $\bP_2$, when $t_1, t_2 >>0$ then $\FS_1 \cap \FS_2 = \emptyset$.
  \end{itemize}
\end{prop}

When $U$ and $W$ are chosen to be relatively compact open semi-algebraic subsets of
$N_P$ and $M_P K$ respectively then the Siegel set $\FS= U \times A_{P, t}
  \times W$ is semi-algebraic in $G$. As the projection map $G \lo
  G/M$ is semi-algebraic, a Siegel set for $G/M$ image of a
  semi-algebraic Siegel set for $\G$ is semi-algebraic. {\em We will
only consider such semi-algebraic Siegel sets in the rest of the text.}

\subsubsection{The partial Borel-Serre compactifications $\ol{G}^{BS}$
  and $\ol{\Gamma \backslash G}^{BS}$} \label{atlas}
 
Let $\bP \subset \G$ be a parabolic subgroup. Let $\Delta= \{ \alpha_1,
\ldots, \alpha_r\}$ be the set of simple roots in $\Phi(A_P,
N_P)$. Consider the semi-algebraic diffeomorphism $e_P: A_P \lo (\RR_{>0})^r$ defined
by  \begin{equation} \label{e_P}
e_P (a) = (a^{-\alpha_{1}}, \ldots, a^{-\alpha_{r}}) \in (\RR_{>0})^r
\subset \RR^r\;\;.
\end{equation}
Let $\ol{A_P} = [0, \infty)^r \subset \RR^r$ be the closure of
$e_P(A_P)$ in $\RR^r$. We denote by $\ol{A_{P,t}} \subset \ol{A_P}$ the closure of $e_P(A_{P,t})$. 

\medskip 
Let $$\ol{G}^{BS}= G \cup \coprod_{\bP \subset \G} (N_P \times (M_P
K))$$ be the Borel-Serre partial compactification of $G$ 
constructed in \cite[\textsection 3.2]{BJ}. The topology on
$\ol{G}^{BS}$ is such that an unbounded sequence $(y_j)_{j \in \NN}$ in
$G$ converges to a point $(n, m) \in N_P \times (M_P K)$
if and only if, in terms of the horospherical decomposition $G = N_P
\times A_P \times (M_PK)$, $y_j = (n_j, a_j, m_j)$ with $n_j \in N_P$,
$a_j \in A_P$, $m_j \in M_P K$, and the components $n_j$, $a_j$ and $m_j$
satisfy the conditions:

1) For any $\alpha \in \Phi(A_P, N_P)$, $(a_j)^\alpha \lo + \infty$,

2) $n_j \lo n$ in $N_P$ and $m_j \lo m$ in $M_P K$.

We refer to \cite[p274-275]{BJ} for the precise description of the
similar glueing between $N_P \times (M_P K)$ and $N_Q \times (M_Q K)$ for two
different parabolic subgroups $\bP \subset \bQ$.

Then:
\begin{prop} \cite[Prop.3.3]{BJ}
The embedding $N_P \times A_P \times (M_PK) = G \subset
\ol{G}^{BS}$ extends naturally to an embedding $N_P \times \ol{A_P}
\times (M_P K) \hookrightarrow \ol{G}^{BS}$.
\end{prop}
We denote by $G(P)$ the image of $N_P \times \ol{A_P}
\times (M_P K)$ under this embedding. It is called the corner
associated with $\bP$. As explained in \cite[Prop. 6.3]{BJ}
$\ol{G}^{BS}$ has the structure of a real-analytic manifold with
corners, a system of real analytic neighbourhood of a point $(n, m) \in N_P \times
(M_PK)$ being given by the $\ol{\FS_{U,t, W}}:= U \times \ol{A_{P, t}} \times W$, for $U$ 
a neighborhood of $n$ in $N_P$, $W$ a neighborhood of $m$ in $M_PK$
and $t>0$, see \cite[Lemma 3.10, Prop. 6.1 and Prop. 6.3]{BJ}.

The left $\G(\QQ)$-multiplication on $G$ extends to a real
analytic action on $\ol{G}^{BS}$: see \cite[Prop. 3.12]{BJ} for the
extension to a continous action and the proof of \cite[Prop. 6.4]{BJ}
for the proof that the extended action is real analytic). The
restriction of this extended action
to a neat $\Ga$ is free and properly discontinuous (see
\cite[Prop. 3.13 and Prop. 6.4]{BJ}). Then $\ol{\Gamma \backslash
  G}^{BS}:= \Gamma \backslash \ol{G}^{BS}$ is a compact real analytic
manifold with corners compactifying $\Gamma \backslash G$. The action
of any compact subgroup $M$ of $K$ on $\Gamma \backslash G$ extends to
a proper real analytic action on $\Gamma \backslash G$, hence the
quotient $\ol{S_{\Ga, G, M}}^{BS} := \Gamma \backslash \ol{G}^{BS} /M$
is a compactification of $ S_{\Ga, G, M}$ in the category of real
analytic manifolds with corners. We denote by $\ol{\pi} : \ol{G}^{BS} \lo \ol{S_{\Ga, G, M}}^{BS} $
the extension of $\pi$.

By \Cref{reduction}(1) there exist finitely many $\bP_1, \ldots, \bP_k$
parabolic subgroups of $\G$ and Siegel sets $\FS_i:= U_i \times A_{P_{i}, t_{i}}
  \times W_i$, $1 \leq
  i \leq k$, with $U_i$, $W_i$ compact semi-algebraic subsets of
$N_{P_{i}}$ and $M_{P_{i}} K$ respectively, whose images $V_i:=
\pi(\FS_{i})$, $1 \leq
  i \leq k$ cover $S_{\Ga, G, M}$.  Then the images $\ol{V_i} :=
\ol{\pi}(\ol{\FS_{i}})$ of $\ol{\FS_{i}}:= U_i \times \ol{A_{P_{i}, t_{i}}}
\times W_i \subset \ol{\G}^{BS}$, $1 \leq
  i \leq k$, cover $\ol{S_{\Ga, G, M}}^{BS}$.
The $\ol{V_i}$'s, $1 \leq i \leq k$, form an explicit finite atlas of the
$\RR_\an$-definable manifold with corners $\ol{S_{\Ga, G, M}}^{BS}$, and their open $\RR_{an}$-definable subsets $V_i$, $1
\leq i \leq k$, form an explicit finite atlas of the
$\RR_\an$-definable submanifold $S_{\Ga, G, M}$ of $\ol{S_{\Ga, G, M}}^{BS}$.

\subsubsection{Proof of \Cref{definability}(2)}

\begin{proof}[\unskip\nopunct]
Let $\FS \subset G/M$ be a Siegel set. Translating $\FS$ by an element
of $\Ga$ (notice that this translation is a semi-algebraic
transformation of $G/M$) we can assume without loss of generality that $\FS$ is (the
image in $G/M$ of) one of the $\FS_i$ of the previous paragraph. We
are thus reduced to showing that for each $i$, $1 \leq i \leq k$, the map $\pi_i: \FS_i \lo V_i $ is $\RR_{\an}$-definable, or equivalently that the composite map $\FS_i \stackrel{\pi_i}{\lo} V_i \hookrightarrow
\ol{V_i}$ is. 

This composite map
factorises as 
$ 
\xymatrix@1{
\FS_i  \ar[rr]^{1_{U_{i}} \times e_{P_{i}}
    \times 1_{W_{i}}} & &\ol{\FS_{i}}\ar[rr]^{\ol{\pi_i}} & &\ol{V_i}\;\;.
}
$
On the one hand, it follows from the definition~(\ref{e_P}) of
$e_{P_{i}}: A_{P_{i}, t_{i}}  \lo \ol{A_{P_{i}, t_{i}}} \subset \RR^{r_{i}}$ that $e_{P_{i}}$, hence also $1_{U_{i}} \times e_{P_{i}}
    \times 1_{W_{i}}$, is semi-algebraic. On the other hand, the map $\ol{\pi_i}: U_i \times
\ol{A_{P_{i}, t_{i}}} \times W_i \lo \ol{V_i}$ is a real analytic map
between compact sets, hence is $\RR_\an$-definable. This concludes the
proof that $\pi_{|\FS}: \FS \lo S_{\Ga, G, M}$ is
$\RR_\an$-definable. 

With the notations above, the set $\cF:=
\cup_{i=1}^k \FS_i$ is a semi-algebraic fundamental set for the action
of $\Ga$ on $G/M$. As each $\pi_i: \FS_i \lo S_{\Ga, G, M} $ is
definable, it follows that $\pi_{\cF}: \cF \lo S_{\Ga, G, M}$ is
$\RR_{\an}$-definable, which concludes the proof of \Cref{definability}(2).
\end{proof}

\subsection{Morphisms of arithmetic quotients are definable: proof of
  \Cref{definability}(3)}

\begin{proof}[\unskip\nopunct]

Let $f: S_{\Gamma', G', M'} \lo S_{\Gamma,G, M}$ be a morphism of
arithmetic quotients. Hence $f$ is deduced from a morphism $f: \G' \lo \G$
of semi-simple linear algebraic $\QQ$-group such that $f(M') \subset
M$ and $f(\Ga') \subset \Ga$. 

In the case where $\G' = \G$, Goresky and MacPherson show in \cite[Lemma 6.3]{GMcP}
that $f$ extends uniquely to a real analytic morphism $\ol{f} :  \ol{S_{\Gamma', G', M'} }^{BS} \lo \ol{S_{\Gamma,G,
    M}}^{BS}$. This proves \Cref{definability}(3) for $\G'=\G$.

\begin{rem}
Notice this implies that any Hecke correspondence $c_g = (c_1,
c_2)$ as in (\ref{Hecke}) has a unique real analytic extension 
\begin{equation} \label{extended Hecke}
\xymatrix{
\overline{c_g} = (\ol{c_1}, \ol{c_2}) :& \ol{S_{\Ga, G, M}}^{BS} &\ol{S_{g^{-1}\Gamma g \cap \Gamma, M}}^{BS} \ar[l]^{\ol{c_1}}
\ar[r]^{g \cdot} \ar@/_1pc/[rr]_{\ol{c_{2}}} & \ol{S_{\Gamma \cap g \Gamma g^{-1}, G, M}}^{BS}
\ar[r] &\ol{S_{\Ga, G, M}}^{BS}}\;\;,
\end{equation}
mapping boundary to boundary (\cite[Lemma 6.3]{GMcP} states only the
continuity of $\ol{c_g}$ but
the real analyticity follows from the proof of \cite[Lemma
6.3]{GMcP}). In particular Hecke correspondences on $S_{\Gamma,G, M}$
are $\RR_\an$-definable.
\end{rem}

Suppose that $f: \G' \lo \G$ is surjective. Without loss of generality
we can assume that $\G$, and then $\G'$, are adjoint. Then $\G' = \G
\times \bH$, $S_{\Ga', G', M'} = S_{\Ga' \cap G, G, M' \cap G} \times S_{\Ga' \cap H, H, M' \cap H}$, the map $f$ coincides with the
projection onto the first factor (this projection is obviously
$\RR_\an$-definable) composed with the morphism of
arithmetic quotients $i: S_{\Ga' \cap G,
  G, M' \cap G}  \lo S_{\Ga, G, M}$, which is $\RR_\an$-definable from the
case $\G' = \G$. This proves \Cref{definability}(3) when $f: \G' \lo \G$ is surjective.

We are thus reduced to proving \Cref{definability}(3) in the case
where $f: \G' \lo \G$ is a strict inclusion. In that case the morphism $f$ does not usually extend
to a real analytic morphism (or even a continuous one)
$\ol{f}$ between the Borel-Serre compactifications (in other words the
Borel-Serre compactification is not functorial). The problem is
that two parabolic subgroups $\bP_i \subset \G$, $i=1,2$, can be non conjugate
under $\Ga$ while their intersections $\bP_i \cap \G'$ are
$\Ga'$-conjugate parabolic subgroups of $\G'$.

Let $(\ol{V'_i})_{1\leq i \leq k}$ be an
$\RR_{\an}$-atlas for $\ol{S_{\Ga', G', M'}}^{BS}$ as in
\Cref{atlas}. Showing that $f: S_{\Ga', G', M'} \lo S_{\Ga, G, M}$ is
$\RR_{\an}$-definable is equivalent to showing that for each
$i$, $1 \leq i \leq k$, the restriction $f: V'_i \lo S_{\Ga, G, M}$ is definable.
As the diagram
$$
\xymatrix{
\FS_i':= U'_{i} \times A'_{P'_{i}, t'_{i}} \times W'_{i} \ar[r]^>>>>>>f
\ar[d]_{\pi'_{i}} & G \ar[d]^{\pi} \\
V'_{i} \ar[r]_f  & S_{\Ga, G, M}
}
$$
is commutative, it is enough to show that the composite
\begin{equation}  \label{composite}
\FS'_i
\stackrel{f}{\longrightarrow} G \stackrel{\pi}{\longrightarrow}
S_{\Ga, G, M}
\end{equation}
is $\RR_{\an}$-definable.

We use the following:
\begin{theor} (\cite[Theor.1.2]{Orr}) \label{orr}
Let $\G$ and $\bH$ be reductive linear $\QQ$-algebraic groups, with
$\bH \subset \G$. Let $\FS_H:= U_H \times A_{P_{H}, t} \times W_H \subset \HH(\RR)$ be a Siegel set
for $\bH$. 

Then there exists a finite set $C \subset \G(\QQ)$ and a Siegel set
$\FS:= U \times A_{P, t} \times W \subset \G(\RR)$ such that $\FS_H \subset C\cdot \FS$.
\end{theor}

Applying this result to $\G' \subset \G$ and the Siegel set $\FS'_i$
of $\G'$, there exists a finite set $C_i \subset \G(\QQ)$
and a Siegel set $\FS_i:= U_i \times A_{P_{i}, t_{i}} \times W_i$ such that the
composition~(\ref{composite}) factorizes as
\begin{equation} \label{factor}
\FS'_i \lo C_i \cdot \FS_i \stackrel{\pi}{\longrightarrow}
S_{\Ga, G, M}\;\;.
\end{equation}

The inclusion $\FS'_i \lo C_i \cdot \FS_i$ is semi-algebraic. The map
$C_i \cdot \FS_i \stackrel{\pi}{\longrightarrow}
S_{\Ga, G, M}$ is $\RR_{\an}$-definable by
\Cref{definability}(2). 

This finishes the proof of \Cref{definability}(3).
\end{proof}

\section{Definability of the period map}

\subsection{Reduction of \Cref{definability period map}
  to a local statement}

In the situation of \Cref{definability period map}, 
let $S \subset \ol{S}$ be a smooth compactification such that $\ol{S}
-S$ is a normal
crossing divisor. Let $(\overline{S_i})_{1\leq i \leq m}$ be a finite
open cover of $\ol{S}$ such that the pair $(\ol{S_{i}}, S_i := S\cap
\overline{S_i})$ is biholomorphic to $(\Delta^n, (\Delta^*)^{r_{i}}
\times \Delta^{l_i:= n-r_{i}})$. To show that the period map $\Phi_S: S \lo \Hod^0(S,
\VV) = S_{\Ga, G, M}$ is $\RR_{\an, \exp}$-definable, it is enough to
show that for each $i$, $1 \leq i \leq m$, the restricted period map
\begin{equation} 
{\Phi_S}_{|S_{i}}: S_i= (\Delta^*)^{r_{i}} \times \Delta^{l_{i}} \lo S_{\Ga, G,
  M}
\end{equation}
is $\RR_{\an, \exp}$-definable. Without loss of generality we can
assume that $r_i= n$ and $l_i=0$ by allowing some factors with trivial
monodromies. Finally we are reduced to proving:

\begin{theor} \label{polydisc}
Let $\VV \rightarrow (\Delta^*)^n$ be a polarized variation of pure
Hodge structures of weight $k$
over the punctured polydisk $(\Delta^*)^n$, with period map $\Phi:
(\Delta^*)^n \lo S_{\Ga, G, M}$.
Then $\Phi$ is $\RR_{\an, \exp}$-definable.
\end{theor}

\subsection{Proof of \Cref{polydisc} assuming
  \Cref{finitely many Siegel}}

\begin{proof}[\unskip\nopunct]
Let us fix $x_0$ a basepoint in  $(\Delta^*)^n$. We denote by $V_\ZZ$ the fiber
$\VV_{x_{0}}$ of $\VV$ at $x_0$ (modulo torsion) and define $V:= V_\ZZ
\otimes_\ZZ \QQ$.
It follows from Borel's monodromy
theorem \cite[Lemma (4.5)]{Schmid} that the monodromy transformation $T_i \in
\G(\ZZ) \subset \GL(V_\ZZ)$, $1\leq i \leq n$, of the local system $\VV$, corresponding to counterclockwise simple
circuits around the various punctures, are
quasi-unipotent. Replacing $(\Delta^*)^n$ by a
finite \'etale cover if necessary we can assume without loss of
generality that all the $T_i$'s are unipotent. Let $N_i \in
\Fg_\QQ$, $1 \leq i \leq n$, be the logarithm of $T_i$; this is a nilpotent element in $\Fg_\QQ$.

Let $\FH$ denote the Poincar\'e upper-half plane and $\exp(2 \pi i
\cdot): \FH \lo \Delta^*$ the uniformizing map of $\Delta^*$. Let $\FS_\FH
\subset \FH$ be the usual Siegel fundamental set $\{ (x, y)
\in \FH\; | \; y>1, -1/2 < x< 1/2\}$. 
Consider the commutative diagram
$$ 
\xymatrix{
\FS_\FH^{n}  \subset \FH^{n}
\ar[d]_{p= \exp(2 \pi i \cdot)^{n}}
\ar[r]^>>>>>>{\tilde{\Phi}} & D= G/M \ar[d]^{\pi} \\
(\Delta^*)^{n} \ar[r]_{\Phi} & S_{\Ga, G,M}\;\;, 
}
$$
where $\tilde{\Phi}$
is the lifting of $\Phi$ to the universal cover.
As the restriction $\exp(2 \pi i \cdot)_{|\FS_{\FH}}$ is $\RR_{\an,
  \exp}$-definable, the map $p_{|\FS_\FH^{n}}|: {\FS_\FH}^{n} \lo
(\Delta^*)^{n}$ is $\RR_{\an,
  \exp}$-definable. We are reduced to proving that the composition
$
\xymatrix@1{{\FS_\FH}^{n} \ar[r]^{\ti{\Phi}} & G/M \ar[r]^\pi
&S_{\Ga, G,M}}$ is $\RR_{\an, \exp}$-definable.

\begin{lem} \label{l1}
The map $\tilde{\Phi} : {\FS_\FH}^{n}\lo G/M$
is $\RR_{\an, \exp}$-definable.
\end{lem}

\begin{proof}
The nilpotent orbit theorem \cite[(4.12)]{Schmid} states that
(after maybe shrinking the polydisk) the map
$\tilde{\Phi}: {\FS_\FH}^{n}  \lo G/M$ is of
the form $$\tilde{\Phi} (z) = \exp (\sum_{j=1}^{n}  z_j N_j) \cdot \Psi(
p( z))$$
for $\Psi: \Delta^n \lo \check{D}$ a
holomorphic map and $\check{D} \supset D$ the compact dual of $D$. The
map $\Psi$ is the restriction to a relatively compact set of a real
analytic map. As $p_{|{\FS_\FH}^{n}}: {\FS_\FH}^{n} \lo
\Delta^n$ is $\RR_{\an, \exp}$-definable, it follows that $(z \mapsto
\Psi(p(z))$ is $\RR_{\an, \exp}$-definable.

The action of $\G(\CC)$ on $\hat{D}$ is definable, $D$ is a
semi-algebraic subset of $\hat{D}$ and $\exp (\sum_{j=1}^n  z_j
N_j)$ is polynomial in the $z_j$'s as the $N_j$'s are nilpotent. 

Hence the result.
\end{proof}

\begin{rem}
Notice that \Cref{l1} appears also in \cite[Lemma 3.1]{BaT}.
\end{rem}

It moreover follows from \Cref{finitely many Siegel}, proven in the next
section, that there exist finitely many Siegel sets $\FS_i$ $(i \in I)$
for $G/M$ such that $\tilde{\Phi}({\FS_\FH}^{n}) \subset \bigcup_{i
  \in I}\FS_i$. As $\pi_{|\FS_i}: \FS_i \lo
S_{\Ga, G, M}$, $i \in I$, is $\RR_{\an}$-definable by 
\Cref{definability}(2), and the set $I$ is finite, we deduce from \Cref{l1} that 
$ \pi \circ \tilde{\Phi} : {\FS_\FH}^{n}
\lo S_{\Ga, G,M}$
is $\RR_{\an, \exp}$-definable. This concludes the proof of
\Cref{polydisc}, hence of
\Cref{definability period map}, assuming \Cref{finitely many Siegel}.
\end{proof}

\subsection{Reminders on the ${SL_2}^n$-orbit theorem} \label{asymptote}

The proof of \Cref{finitely many Siegel} will follow from the
understanding of the asymptotic behavior of the local period map $\Phi:
(\Delta^*)^n \lo S_{\Ga, G, M}$. Let $F:= \Psi(0) \in \check{D}$ be
the {\em limiting Hodge filtration}. We denote by  
$$ \theta (\bz) = \exp(\sum_{j=1}^n z_j N_j ) \cdot F$$ the
corresponding {\em nilpotent orbit} $\theta: \CC^n \lo \check{D}$ in the sense of
\cite[(1.14)]{CKS}. The Nilpotent Orbit Theorem of Schmid
states that $\theta$ approximates $\tilde{\Phi}$
in a sharp way (see \cite[(4.12)]{Schmid} and the refined
\cite[(1.15)]{CKS} for the precise statement).  The ${SL_2}^n$-orbit
theorem of \cite{Schmid}, \cite{Ka}, \cite{CKS} and
\cite{CKS87} approximates the nilpotent orbit itself by finitely many ${SL_2}^n$-orbits.

In this section we recall what we need from the ${SL_2}^n$-orbit
theorem. We continue with the notations used in the proof of
\Cref{polydisc} and introduce the notations we need from
\cite{CKS}. Notice that the theory developed in \cite{CKS} (hence the
content of this section) works for any nilpotent orbit, not
necessarily rational. The rationality will be used in the next
section only.

Recall that if $N$ is a nilpotent operator on a vector space $V$ over
a field of characteristic zero, there exists a unique increasing
filtration $W:= W(N)$ of $V$, called the {\em monodromy filtration of
  $N$}, such that $N(W_l) \subset W_{l-2}$ and $N: \Gr^W_l \lo
\Gr^W_{-l}$ is an isomorphism \cite[(1.6.1)]{De74}. Let $C:= \{\sum_{i=1}^n \lambda_j N_j,
\lambda_j \in \RR_{>0}\}$. In \cite[(3.3)]{CK82} Cattani and Kaplan prove
that every $N \in C$ defines the same filtration $W(C):= W(N)$. Moreover, for $J$ a subset of $\{1, \cdots, n\}$, let
$C_J:= \{ \sum_{j \in J} \lambda_j N_j, \lambda_j \in \RR_{>0}\}$ be a
facet of the closure $\ol{C}$; then for every $N \in C_J$ and $N' \in
C_{J'}$, for $N'' \in C_{J \cup J'}$, $W(N'')$ is the weight
filtration of $N$ relative to $W(N')$ in the sense of
\cite[(1.6.13)]{De74}: $N$ preserves $W(N')$, $N \cdot
W(N'')_l \subset W(N'')_{l-2}$ and the filtration induced by $W(N'')$
on $\Gr^{W(N')}_l$ is the monodromy filtration of $N$ on
$\Gr^{W(N')}_l$ centered at $l$ (i.e. $N^k: \Gr^{W(N'')}_{l+k} \Gr^{W(N')}_l
\simeq \Gr^{W(N'')}_{l-k} \Gr^{W(N')}_l$). 

The $SL_2$-orbit theorem in one variable of Schmid \cite{Schmid} implies that $(W(C_J)[-k],
F)$ is a mixed Hodge structure polarized by any element $N$ of $C_J$ (in
the sense of \cite[(2.26)]{CKS}). The ${SL_2}^n$-orbit 
\cite[Theor. 4.20]{CKS} states a deep compatibility between these various mixed
Hodge structures when $J$ varies {\em following a fixed ordering of
the variables in $(\Delta^*)^n$}, which implies that the nilpotent orbit $\theta$ can be
sharply approximated by an ${SL_2}^n$-orbit on a sector of
$(\Delta^*)^n$ associated to this ordering. 

Let $\sigma$ be an ordering of the variables in $\CC^n$. For $1 \leq r \leq
n$ we shall follow the convention in \cite{CKS} and use a bold index $\br$ to label objects associated to the
$r$-th first variables with respect to $\sigma$. Let $C_{\br}:= \{ \sum_{j =1}^r \lambda_j N_j, \lambda_j \in
\RR_{>0}\}$ and let $W^{\br}:= W(C_\br)[-k]$. One defines Hodge filtrations $\tilde{F}_{\br} \in
\check{D}$ for $1 \leq r \leq n$ by descending induction on $r$ as
follows.
We define $\ti{F}_{\bn}$ to be the Hodge filtration of the $\RR$-split mixed Hodge
structure canonically associated to $(W^\bn, F)$ by the
$SL_2$-orbit theorem in one variable \cite[(3.30)]{CKS}. The
$\RR$-split mixed Hodge structure $(W^\bn, \ti{F}_\bn)$ is
polarized by every $N \in C_\bn$, hence the mapping $$(z_1, \cdots,
z_{n-1}) \mapsto \exp(\sum_{j=1}^{n-1} z_j N_j) \cdot (e^{i N_n}
\tilde{F}_\bn)$$ is a nilpotent orbit in $(n-1)$-variables, which takes
values in $D$ for $\Ima(z_j)>0$. In particular $(W^{\mathbf{n-1}}, e^{i N_n}
\tilde{F}_{\bn})$ is a MHS polarized by every $N \in C_{\mathbf{n-1}}$. Let
$\tilde{F}_{\mathbf{n-1}}$ be the Hodge filtration of the $\RR$-split MHS
canonically associated to it by the $SL_2$-orbit theorem in one
variable. Continuing this way one obtains, for each $1\leq r \leq n$, $\tilde{F}_\br \in \check{D}$
such that $(W^\br, \tilde{F}_\br)$ is the $\RR$-split
MHS canonically associated to $(W^{\br}, e^{i N_{r+1}}
\tilde{F}_{\mathbf{r+1}})$. It is polarized by every $N \in C_\br$.

For each $1 \leq r \leq n$ let $\hat{Y}_{\br} \in \Fg_\RR$ be the semi-simple
endomorphism grading the $\RR$-split MHS $(W^\br, \tilde{F}_\br)$
in the sense of \cite[(2.27)]{CKS}: if $V_\CC = \bigoplus I^{p,q}$ is
the Deligne's splitting \cite[(2.13)]{CKS} of  $(W(C_\br)[-k],
\tilde{F}_\br)$ (hence $I^{p,q} = \ol{I^{q,p}}$ as it is $\RR$-split)
then $\hat{Y}_{\br}$ act by multiplication by $l -k$ on the real subspace $V_l^{\br}:= \bigoplus_{p+q =l}
  I^{p,q}$. A crucial fact is that the endomorphisms
  $(\hat{Y}_{\br})_{1 \leq r \leq n}$ commute (see
  \cite[(4.37)(ii)]{CKS}), hence they define a
real multi-grading $V_\RR = \bigoplus_{l_1, \cdots, l_n} V_{l_{1},
  \cdots, l_{n}}$, where $V_{l_{1},\cdots, l_{n}} := V_{l_{1}}^{\mathbf{1}} \cap
\cdots V_{l_{n}}^{\bn} \simeq
\Gr^{W^{\bn}}_{l_{n}}(\Gr^{W^{\mathbf{n-1}}}_{l_{n-1}}(\cdots(\Gr^{W^{{\mathbf{1}}}}_{l_{1}})
\cdots))$ (denoted $U_{l_{1},\cdots, l_{n}}$ in \cite[(2.18)]{CKS87}).

For each $1\leq j \leq n$, let $\hat{N}_j^- \in \Fg_\RR$ be the
component of $N_j$ in the subspace $\bigcap_{r=1}^{j-1} \ker (\ad\,
\hat{Y}_{\br})$ relative to the decomposition of $\Fg_\RR$ in
eigenspaces of the commuting set of semisimple endomorphisms $(\ad\,
\hat{Y}_{\br})_{1 \leq r \leq j-1}$. 

For $1 \leq j \leq n$ define $\hat{Y}_j \in \Fg_\RR$ by the system of linear
equations (see \cite[(4.18)]{CKS})
\begin{equation} \label{relation}
\forall\, 1 \leq r \leq n, \quad \hat{Y}_{\br} = \sum _{j=1}^r
\hat{Y}_j \;\;.
\end{equation}
The pair $(\hat{N}_j^-,
\hat{Y}_j)$ satisfies $[\hat{Y}_j, \hat{N}_j^-] = -2 \hat{N}_j^-$
hence by the Jacobson-Morosov theorem there exists a unique element
$\hat{N}_j^+ \in \Fg_\RR$ such that $(\hat{N}_j^-, \hat{Y}_j,
\hat{N}_j^+)$ is an $\mathfrak{sl}_2$-triple. This triple defines a Lie
algebra homomorphism
$$ \rho_{j, *}: \mathfrak{sl}_{2, \RR}  \lo \Fg_\RR\;\;.$$
The ${\SL_2}^n$-orbit theorem \cite[(4.20)]{CKS} claims that the triples $(\hat{N}_j^-, \hat{Y}_j,
\hat{Y}^+)_{1 \leq j \leq n}$ are commuting two by two, hence there is 
unique algebraic group homomorphism
\begin{equation}  \label{morph}
\rho: {\SL_{2,\RR}}^n \lo \G_\RR
\end{equation} such that $\rho_{j, *}$ coincide with the
restriction to the $j$-th factor of the differential $\rho_*$ of
$\rho$; moreover the $\SL_2(\RR)^n$-orbit passing through
$\exp^{iN_1}\tilde{F}_1$ approximates sharply the nilpotent orbit $\theta$ on any
domain $C_{\sigma, \eta, \varepsilon}:= \{ \bz := \bx + i \by \in
\FH^n\;|\; y_1 \geq \varepsilon y_2, \ldots, y_{n-1}\geq \varepsilon y_n \,;\,
y_n \geq \eta \}$. For the convenience of the reader we recall the
precise statement of the second half of \cite[(4.20)]{CKS}:

\begin{theor} \cite[(4.20)(v)-(viii)]{CKS} \label{multiorbit}
 There exist $G$-valued functions $g_{\br}(y_1, \cdots, y_t)$ defined
 for $y_j>0$ if $1 \leq r \leq n-1$ and for $y_j>\alpha$, $\alpha \in
 \RR$ if $r=n$ such that:

\begin{itemize}
\item[(v)] For $j \leq r \leq n$, $g_\bj(y_1, \cdots, y_j)$ is  a
  $(0,0)$-morphism of the MHS $(W^{\br}, \tilde{F}_{\br})$.
\item[(vi)] $\sum_{s=1}^r y_s N_s = \Ad (\prod_{j=r-1}^1 g_\bj
  (y_{1}/y_{j+1}, \cdots, y_{j}/y_{j+1})) \sum_{r=1}^r
  y_s \hat{N}_s^-$.
\item[(vii)] $\exp(i \sum_{j=1}^n y_j N_j) \cdot F = (\prod_{r=n}^1
  (g_\br(y_{1}/y_{r+1}, \cdots, y_{j}/y_{r+1}))) \exp(
  i \sum_{j=1}^n y_j \hat{N}_j^-) \cdot \tilde{F}_\bn$.
\item[(viii)] The functions $g_\br(y_1, \cdots, y_r)$ and $g_\br(y_1,
  \cdots, y_r)^{-1}$ have power series expansions in non-positive
  powers of $y_1/y_2$, $y_2/y_3$, ...., $y_r$ with constant term $1$
  and convergent in regions of the form $C_{\sigma, \eta,
    \varepsilon}$.
\end{itemize}
\end{theor}

\subsection{Proof of \Cref{finitely many Siegel}}

\begin{proof}[\unskip\nopunct]
Let $R>0$, $\eta >0$ and $\Phi: (\Delta^*)^n\lo S_{\Ga, G, M}$, be as in \Cref{finitely
  many Siegel}. 
For any point $\bz \in \FH$ with $|\Reel \bz| \leq R$ and $\Ima \bz
>\eta$ there exists an ordering $\sigma$ of the variables in
$(\Delta^*)^n$ such that $\bz \in C_{\sigma, \eta, 1}$. As any compact
set is contained in a Siegel set, 
\Cref{finitely many Siegel} follows from \Cref{Un Siegel} below.
\end{proof}

\begin{theor} \label{Un Siegel}
Let $\Phi: (\Delta^*)^n \lo S_{\Ga, G, M}$ be a period
map with unipotent monodromy on a product of punctured disks, with a
given order of variables $\sigma$. 

For any given constants $R>0$, and $\eta >0$ sufficiently large, there exists a Siegel set
$\FS \subset G/M$ such that $\tilde{\Phi}(\bz) \subset  \FS$
whenever $|\Reel \bz| \leq R$ and $\bz \in C_{\sigma, \eta, 1}$. 
\end{theor}

\begin{proof}[Proof of \Cref{Un Siegel}]
When $n=1$ this is proven by Schmid in \cite[(5.29)]{Schmid}. We
adapt his proof for any $n \geq 1$.

\subsubsection{} The first step consists in 
generalizing \cite[(5.17)]{Schmid}, using that the $N_i$'s, $1 \leq i \leq n$, are
rational:

\begin{lem} \label{rational}
There exists a morphism of $\QQ$-groups
$$\psi: {\SL_2}^n \lo \G$$ and an element $g \in G$ such that
$\psi_\RR = g \circ \rho \circ g^{-1}$ (where $\rho$ is defined in \Cref{morph}).
\end{lem}

\begin{proof}
This is the content of \cite[(3.7.1)]{CDK95}, where it is shown that
there exists an element $g \in G$ transporting the multi-grading
$V_{l_{1}, \cdots, l_n}$ of $V_\RR$ splitting all the $W^{\bj}$ to a
rational one. Notice that Cattani, Deligne and Kaplan argue with the
multi-grading obtained from the commuting semi-simple elements
$(\hat{Y}_j)_{1\leq j \leq n}$ (see \cite[(3.5.2)]{CDK95}) rather than
the one we use obtained from the $(\hat{Y}_{\bj})_{1\leq j \leq n}$
but this is clearly equivalent in view of \Cref{relation}.
\end{proof}

{\em From now on, and in order to keep the notations reasonable, we
replace $\hat{N}_j^{-}$, $\hat{Y}_j$, etc ... by $\Ad g (\hat{N}_j^-)$, $\Ad g
(\hat{Y}_j)$, 
etc ... so that the $SL_2$-triplets $(\hat{N}_j^{-}, \hat{Y}_j,
\hat{N}_{j}^+)$ belongs to $\Fg_\QQ$}.
With these new notations Theorem \cite[(4.20)]{CKS}
remains true if one performs the obvious notational modifications. In particular
\cite[(4.20)(viii)]{CKS} becomes:
$$
\exp(i \sum_{j=1}^n y_j N_j) \cdot F = g_\sigma(\bss) 
  g^{-1} \cdot \exp( i \sum_{j=1}^n y_j \hat{N}_j^-) \cdot
  g\tilde{F}_\bn\;\;,
$$
where $s_1 := y_1/y_2$, ... , $s_{n-1}:= y_{n-1}/y_n$, $s_n=
y_n$, $\bss:= (s_1, \cdots, s_n)$ and $g_\sigma(\bss):= (\prod_{r=n}^1
  (g_\br(y_{1}/y_{r+1}, \cdots, y_{j}/y_{r+1})))$. 
In view of \cite[(3.12)]{CKS} and the relation $e^{i
  \hat{N}_\br^-} \tilde{F}_\br = e^{i N_1}\tilde{F}_{\mathbf 1}$, see \cite[(4.20)(i)]{CKS},
this equality becomes
\begin{equation} \label{approx}
\exp(i \sum_{j=1}^n y_j N_j) \cdot F = g_\sigma(s) g^{-1} \cdot \exp(-
\frac{1}{2} \sum_{j=1}^n \log s_j \hat{Y}_{\bj}) \cdot g e^{i N_1}
\tilde{F}_{\mathbf{1}} \;\;.
\end{equation}

\subsubsection{} Using \Cref{rational} we can associate a
rational parabolic $\bP$ of $\G$ to our situation as follows.

\sspace
Let $\Fh:= \psi_* ({\mathfrak{sl}_2}^n) \subset \Fg_\QQ$ be the Lie algebra of the image
$\HH:= \psi({\SL_2}^n) \subset \G$. We denote by $\Fs \subset \Fh$ its
toral subalgebra generated by the semi-simple elements $\hat{Y}_j$'s.  

\begin{lem}
Let $s$ be the dimension of $\Fs$. There exist $1\leq i_1< i_2 <\ldots<
i_s\leq n$ such that $\hat{Y}_{r} =0$ for $r
\not = i_j$ for some $1\leq j \leq s$, and $\Fh \simeq
{\mathfrak{sl}_2}^s$ is the product of the $\mathfrak{sl}_2$-triples $(\hat{N}_{i_{j}}^-, \hat{Y}_{i_{j}},
\hat{N}_{i_{j}}^+)$, $1 \leq j \leq s$.
\end{lem}

\begin{proof}

Suppose that $(\hat{Y}_{i_{j}})_{1\leq j \leq s_l}$ is a
  maximal set of linearly independent vectors among the $\hat{Y}_k$'s,
  $1 \leq k \leq l$.  As $\hat{Y}_{i_{l+1}}$ has to commute with all $\hat{N}_{i_{j}}^-$ for
$1\leq j \leq s_l$ and $[\hat{Y}_{i_{j}}, \hat{N}_{i_{k}}]= -2 \delta_{jk}
\hat{N}_{i_{j}}$, either $\hat{Y}_j$ is linearly independent from the
$(\hat{Y}_{i_{j}})_{1\leq j \leq s_l}$'s, or $\hat{Y}_j=0$
(corresponding to the case $W^j = W^{j+1}$). It follows immediately by
induction on $l$ that there exists $1\leq i_1< i_2 <\ldots<
i_s\leq n$ such that $\hat{Y}_{r} =0$ for $r
\not = i_j$ for some $1\leq j \leq s$; that $\Fs =
\oplus_{j+1}^{s} \QQ \cdot \hat{Y}_{i_{j}}$; and that $\Fh \simeq
{\mathfrak{sl}_2}^s$ is the product of the $\mathfrak{sl}_2$-triples $(\hat{N}_{i_{j}}^-, \hat{Y}_{i_{j}},
\hat{N}_{i_{j}}^+)$, $1 \leq j \leq s$.
\end{proof}

Let $\bS \subset \G$ be the $\QQ$-split torus with Lie algebra
$\Fs$. The basis $(\hat{Y}_{\mathbf{i_{j}}})_{1 \leq j \leq s}$
defines a total order $>$ on $X^*(\bS)$ as follows: a character $\lambda
\in X^*(\bS)$ is positive if its differential 
$d\lambda \in \Fs^*$ is non-zero and satisfies:
\begin{equation} \label{pos}
\forall \, 1\leq j \leq s, \quad d \lambda(\hat{Y}_{\mathbf{i_{j}}})
\leq 0\;\;.
\end{equation}
Notice that~(\ref{pos}) is equivalent to 
\begin{equation} \label{pos1}
\forall \, 1\leq r \leq n, \quad d \lambda(\hat{Y}_{\br})
\leq 0\;\;.
\end{equation}

Let $\Phi^+(\SS, \G) \subset \Phi(\SS, \G)$ be the set of
positive roots for $>$ and let $\bP \subset \G$ be the associated
parabolic subgroup, see \cite[Theor. 4.15]{BT}. Its Levi subgroup is $\LL\subset
\G$, the centralizer in $\G$ of $\SS$. Notice that the Lie algebra
$\Fn_\PP \subset \Fp$ of the unipotent radical $\bN_P \subset \bP$ is the
set of elements in $\Fg$ which belongs to $\bigcap_{r=1}^n W^\br_{\geq
  0}\, \Fg$ and to $W^\br_{>0} \, \Fg$ for at least one index $r$ (where  $W^\br
\Fg$ is the weight filtration induced on $\Fg \subset \End V$ by 
the filtration $W^\br$ on $V$). 

\begin{lem}
The element $g$ in \Cref{rational} belongs to $N_P$.
\end{lem}

\begin{proof}
By definition $g$ preserves $W^\br$, and induces the identity on
$\Gr^{W^{\br}}$, for all $1 \leq r \leq n$. Hence the result.
\end{proof}

\subsubsection{} The second step in the proof of \Cref{Un Siegel}
consists in generalizing \cite[(5.25)]{Schmid}, stating that \Cref{Un
  Siegel} holds true if one replace the local period map by its
approximating nilpotent orbit:

\begin{lem} \label{Un Siegel Nilp}
For any given constants $R>0$, and $\eta >0$ sufficiently large, there exists a Siegel set
$\FS \subset G/M$ such that $\exp(i \sum_{i=1}^n z_i N_i) \cdot F \subset  \FS$
whenever $|\Reel \bz| \leq R$ and $\bz \in C_{\sigma, \eta, 1}$. 
\end{lem}

\begin{proof}
We denote by $K$ the unique maximal subgroup of $G$ containing the stabilizer in $G$ of $e^{i N_1}
\tilde{F}_{\mathbf{1}} \in D$. 
The horospherical decomposition $G= N_P A_P (M_PK)$ implies that there
exists unique real analytic functions $n(s)$, $a(s)$, and $m(s)$ on $\RR_{>0}^n$,
taking value in $N_P$, $A_P$ and $M_PK$ respectively, such that 
\begin{equation} \label{horospherical}
 g_\sigma(s) g^{-1} \cdot \exp(-
\frac{1}{2} \sum_{j=1}^n \log s_j \hat{Y}_{\bj})  = n(s) a(s) m(s) \;\;.
\end{equation}

It follows from~\ref{approx} that 
\begin{equation} \label{approx2}
\exp(i \sum_{j=1}^n z_j N_j) \cdot F = (\exp (\sum_{i=1}^n x_i N_i) \,
n(s)) \cdot a(s) \cdot m(s) \cdot  (g e^{i N_1}\tilde{F}_{\mathbf{1}})\;\;.
\end{equation}

By \cite[(4.20) (viii)]{CKS} the functions $n(s)$, $a(s)$, and $m(s)$
have power series expansions in non positive powers of the
$(s_j)^{1/2}$, $1 \leq i \leq n$, which are convergent on the domain
$C_{\sigma, \eta, 1}$ for $\eta$ sufficiently large, and such that
$n(\infty) = g^{-1}$, $\lim_{s \mapsto \infty} a(s) \exp (\frac{1}{2}
\sum_{j=1}^n \log s_j \hat{Y}_{\bj}) = 1$ and $m(\infty)
=1$. This proves that there exists a Siegel set $\FS \subset G/M$ such that
$\exp(i \sum_{i=1}^n y_i N_i) \cdot F \subset  \FS$ 
whenever $\by \in C_{\sigma, \eta,1}$ with
$\eta$ sufficiently large.

Moreover, it follows from \Cref{nil} below that $\sum_{i=1}^n x_i
N_i$ describes a compact set of $N_P$ when $|\bx| \leq R$. Enlarging
the $U_P$-component of $\FS$ is necessary, it follows that
$\exp(i \sum_{i=1}^n z_i N_i) \cdot F \subset  \FS$ 
whenever $|\Reel \bz| \leq R$ and $\bz \in C_{\sigma, \eta, 1}$ for
$\eta$ sufficiently large. Hence the result.
\end{proof}

\begin{lem} \label{nil}
The element $N_j \in \Fg_\QQ$, $1\leq j \leq n$, belong to the nilpotent radical
$\Fn_\PP$ of the parabolic Lie algebra $\Fp$.
\end{lem}

\begin{proof}
For $j \leq r$  the element $N_j$ is a $(-1, -1)$-morphism of $(W^\br,
\tilde{F}_\br)$, hence $N_j \in W^\br_{-2} \Fg$ (in fact $N_j$ is homogeneous of weight $(-2)$ relative to
$V^r$, i.e. $[\hat{Y}_\br, N_j] = -2 N_j$, see \cite[proof of
(4.42)]{CKS} and \cite[(2.19)(i)]{CKS87}).

For $1 \leq r \leq j-1$, it follows from \cite[(4.20)(vi)]{CKS} or
\cite[(2.19) (ii) and (iv)]{CKS87} that $N_j \in W^\br_{\leq 0} \Fg$.

Hence $N_j$ belongs to $\bigcap_{r=1}^n W^\br_{\leq 0}\Fg$, which
finishes the proof.

\end{proof}

\subsubsection{} Finally, the nilpotent orbit
theorem in many variables \cite[(4.12)]{Schmid} (refined in
\cite[(1.15)]{CKS}) states that for any $G$-invariant distance $d$ on $D$
there exists non-negative constants $\eta$, $\beta$ and $K$ such
that, for $\bz \in C_{\sigma, \eta,1}$, 
\begin{equation}
 d(\tilde{\Phi}(\bz), \exp(i \sum_{j=1}^n z_j N_j) \cdot F) \leq K
 \cdot \sum_{j=1}^n (\Ima z_j)^\beta \exp(-2 \pi y_j) \;\;.
\end{equation} 

\Cref{Un Siegel} for the local period map $\tilde{\Phi}$ follows from
the analogous result \Cref{Un Siegel Nilp} for its
approximating nilpotent orbit $\exp(i \sum_{j=1}^n z_j N_j) \cdot F$
exactly as in the deduction of \cite[(5.26)]{Schmid} from
\cite[(5.25)]{Schmid} in one variable. 

\end{proof}

\subsection{\Cref{definability} implies Borel's algebraicity theorem} \label{algebraic}

\begin{theor}\cite[Theor. 3.10]{Bor72}
Let $S$ be a complex algebraic variety and $f: S \lo S_{\Ga, G, K}$ a
complex analytic map to an arithmetic variety $S_{\Ga, G, K}$. Then
$f$ is algebraic.
\end{theor}

\begin{proof}
The map $f$ is a period map, hence is $\RR_{\an, \exp}$-definable by
\Cref{definability period map}. The graph of $f$ is thus a complex
analytic, $\RR_{\an, \exp}$-definable, subset of the smooth complex
algebraic manifold $S \times S_{\Ga, G, K}$. Recall the
following o-minimal GAGA theorem of Peterzil-Starchenko \cite[Theor. 4.4
and Corollary 4.5]{PS}, generalizing a result of
Fortuna-\L ojasiewicz \cite{FL} in the semi-algebraic case:

\begin{theor}(Peterzil-Starchenko) \label{PS}
Let $S$ be a smooth complex algebraic manifold (hence endowing the
$\CC$-analytic manifold $S$ with a canonical $\RR_\alg$-definable
manifold structure). Let $W \subset S$ be a complex analytic
subset which is also an $\cS$-definable subset for some o-minimal
structure $\cS$ expanding $\RR_{\an}$. Then $W$ is an algebraic subset
of $S$.
\end{theor}

It follows that the graph of $f$ is an algebraic subvariety of $S \times S_{\Ga, G, K}$, hence
that $f$ is algebraic (see \cite[Prop. 8]{Serre}).
\end{proof}

\section{Algebraicity of Hodge loci: proof of \Cref{algebraicity}}

\begin{proof}[\unskip\nopunct]
We refer to \cite[Section 3.1]{K17} for the notions of (connected)
Hodge datum and morphism of (connected) Hodge data, connected Hodge varieties and Hodge
morphisms of connected Hodge varieties. Notice that any connected
Hodge variety is in particular an arithmetic quotient and that any Hodge morphism of
connected Hodge varieties is in particular a morphism of arithmetic
quotients. 

A special subvariety $Y$ of the connected Hodge variety $S_{\Ga, G, M}$ is by
definition the image $Y:= f(S_{\Ga', G', M'})$ of some Hodge morphism $f: S_{\Ga', G', M'} \lo
S_{\Ga, G, M}$. It follows from \Cref{definability}(3) and the remark
above that any special subvariety of $S_{\Ga, G, M}$ is an $\RR_{\an, \exp}$-definable subset of
$S_{\Ga, G, M}$ (endowed with its $\RR_\an$-structure of
\Cref{definability}(1). The Hodge locus $\HL(S_{\Ga, G, M})$ is defined as
the (countable) union of special subvarieties of $S_{\Ga, G, M}$.

The Hodge locus $\HL(S, \VV)$ co\"incides with the preimage $\Phi_S^{-1}
(\HL(S_{\Ga, G, M}))$. Hence to prove \Cref{algebraicity} we are
reduced to proving that the preimage $W:= \Phi^{-1}(Y)$ of any special subvariety $Y
\subset S_{\Ga, G, M}$ is an algebraic subvariety of $S$. By \Cref{definability} the period map $\Phi_S: S \lo  S_{\Ga, G, M}$
is $\RR_{\an, \exp}$-definable. As $Y \subset S_{\Ga, G, M}$ is an $\RR_{\an, \exp}$-definable subset of
$S_{\Ga, G, M}$ it follows that $W=\Phi_S^{-1}(Y)$ is
an $\RR_{\an, \exp}$-definable subset of $S$ (in particular has
finitely many connected components). As $W$ is also a complex analytic
subvariety, the o-minimal GAGA \Cref{PS} of Peterzil-Starchenko
implies that $W$ is an algebraic subvariety of $S$,
which finishes the proof of \Cref{algebraicity}.

\end{proof}

\sspace
\noindent Bruno Klingler : Humboldt Universit\"at zu Berlin

\noindent email : \texttt{bruno.klingler@hu-berlin.de}.

\end{document}